\documentclass{amsart}
\usepackage{amsmath, amsthm, amssymb, mathtools}
\usepackage[margin=0.5in]{geometry}
\usepackage{comment}
\usepackage{parskip}
\usepackage{changepage}
\usepackage{color}
\usepackage[usenames,dvipsnames,svgnames,table]{xcolor}

\specialcomment{remarks}
{

\vspace{2mm}\noindent
\begin{adjustwidth}{5mm}{0mm}
\begingroup\color{blue}
}
{
\endgroup
\end{adjustwidth}

\vspace{2mm}
}

\specialcomment{hide}{}{}
\excludecomment{hide}

\specialcomment{ToBeRemoved}{\begingroup\color{Purple}}{\endgroup}

\begingroup
    \makeatletter
    \@for\theoremstyle:=definition,remark,plain\do{%
        \expandafter\g@addto@macro\csname th@\theoremstyle\endcsname{%
            \addtolength\thm@preskip\parskip
            }%
        }
\endgroup

\newtheorem*{thm*}{Theorem}
\newtheorem{thm}{Theorem}[section]
\newtheorem{cor}[thm]{Corollary}
\newtheorem{lemma}[thm]{Lemma}
\newtheorem{propn}[thm]{Proposition}
\newtheorem*{propn*}{Proposition}

\newtheorem{conj}[thm]{Conjecture}
\theoremstyle{definition}
\newtheorem{defn}[thm]{Definition}

\theoremstyle{remark}
\newtheorem{remark}[thm]{Remark}
\newtheorem*{remark*}{Remark}

\def\NR{\mathit{NR}}
\DeclareMathOperator{\ord}{ord}
\DeclareMathOperator{\dv}{div}
\newcommand{\NZ}{\mathbb{N}_0}
\newcommand{\atext}[1]{\text{\quad #1\quad}}
\newcommand{\NN}{\mathbb{N}}

\title{On Sylvester sums of compound sequence semigroup complements}
\author{T. Alden Gassert}
\address{Department of Mathematics, Western New England University, Springfield, MA 01119}
\email{thomas.gassert@wne.edu}
\author{Caleb McKinley Shor}
\address{Department of Mathematics, Western New England University, Springfield, MA 01119}
\email{cshor@wne.edu}

\begin{document}
\begin{abstract}
In this paper, we consider the set $\NR(G)$ of natural numbers which are not in the numerical semigroup generated by a compound sequence $G$.  We generalize a result of Tuenter which completely characterizes $\NR(G)$.  We use this result to compute Sylvester sums, and we give a direct application to the computation of weights of higher-order Weierstrass points on some families of complex algebraic curves. 
\end{abstract}
\maketitle
\section{Introduction, motivation}
Let $\mathbb{N}$ and $\mathbb{N}_0$ denote, respectively, the sets of positive integers and non-negative integers.  Let $A=(a_1,\dots,a_k)$ and $B=(b_1,\dots,b_k)\in\mathbb{N}^k$ such that $\gcd(a_i,b_j)=1$ for all $i\ge j$.  Let $g_0=\prod_{j=1}^k a_j$ and, for $1\leq i\leq k$, let $g_i=g_{i-1}b_i/a_i$ (i.e. $g_i=b_1\cdots b_i a_{i+1}\cdots a_k$). We say the sequence $(g_i)_{i=0}^{k}$ is a \textit{compound sequence}, and we denote it $G(A,B)$.  Such a sequence can be seen as a generalization of a geometric sequence.  We say a set $G$ is \textit{compound} if its elements can be ordered to form a compound sequence.


Let 
\[R(A,B)=\left\{\sum_{i=0}^k n_i g_i : n_i\in\NZ, g_i\in G(A,B)\right\},\]
the set of integers representable as non-negative linear combinations of elements of $G(A,B)$, and let $\NR(A,B)$ be the complement of $R(A,B)$ in $\NZ$.  That is, $\NR(A,B)$ is the set of positive integers which are not representable --- hence ``$\NR$'' --- as non-negative linear combinations of elements of the compound sequence $G(A,B)$.  It is known that $\NR(A,B)$ is a finite set.  In this paper, we are interested in the sum of the $m$th powers of elements of $\NR(A,B)$; i.e. a formula for the $m$th \textit{Sylvester sum}
\begin{align} \label{eq:sylvester sum}
S_m(A,B):=\sum\limits_{n\in \NR(A,B)}n^m.
\end{align}

If $G$ is a finite geometric sequence of positive integers such that $\gcd(G)=1$, it follows that $G=\{a^{k-i}b^{i} : 0\leq i\leq k\}=G(A,B)$, where $A=(a,\dots,a), B=(b,\dots,b)\in\mathbb{N}^k$ for some relatively prime $a,b\in\mathbb{N}$.
When $G$ is geometric, we write $S_m(a,b;k)$ to denote the $m$th Sylvester sum $S_m(A,B)$.  Any set $G=\{a,b\}$ with $\gcd(a,b)=1$ is compound, and so has corresponding $m$th Sylvester sum $S_m(a,b;1)$. 

\subsection{Known results}
The sum in equation \eqref{eq:sylvester sum} is so named due to its proximity to the Sylvester \emph{denumerant}. Given a non-negative integer $n$ and a set of positive, relatively prime integers $G = \{g_1,\ldots, g_\ell\}$, the denumerant $D(n;G)$ is the number of solutions to $\sum_{i=1}^\ell x_i g_i = n$ in $\mathbb N_0^\ell$. If $a$ and $b$ are relatively prime, then $D(n;\{a,b\}) \in \{0,1\}$ for $0 \le n \le ab-1$, and in \cite{Sylvester1882} Sylvester notes that in this simplest case
\begin{align}\label{eqn:sylvester_result}
S_0(a,b;1)=\sum_{n=0}^{ab-1}(1-D(n;\{a,b\})) = \frac{(a-1)(b-1)}{2}.
\end{align}
The case where $m=1$ and $k=1$ was computed by Brown and Shiue \cite{BrownShiue93}, where they found that 
\begin{align}\label{eqn:brown_shiue_result}
S_1(a,b;1)=\frac{(a-1)(b-1)(2ab-a-b-1)}{12}.
\end{align}
Shortly afterward in \cite{Rodseth93}, using an exponential generating function, R{\o}dseth found for $m\geq1$, 
\begin{align}\label{eqn:rodseth_result}
S_{m-1}(a,b;1)=\frac{1}{m(m+1)}\sum_{i=0}^{m}\sum_{j=0}^{m-i}\binom{m+1}{i}\binom{m+1-i}{j}B_i B_j a^{m-j}b^{m-i}-\frac{1}{m}B_m,
\end{align}
where $B_0, B_1, B_2, \dots$ are the Bernoulli numbers.

In \cite{Tuenter06}, Tuenter presented an identity which characterizes the non-representable numbers for the case where $k=1$.  In particular, for any function $f$ defined on the positive integers, one has 
\begin{align}\label{eqn:tuenter_result}
\sum_{n\in \NR(a,b;1)}\left(f(n+a)-f(n)\right)=\sum_{n=1}^{a-1}\left(f(nb)-f(n)\right).
\end{align}
Among this identity's numerous applications, one can use the monomial $f(n)=n^{m+1}$ to derive an explicit formula for $S_m(a,b;1)$, and the exponential function $f(n)=e^{nz}$ to derive equation~\eqref{eqn:rodseth_result}.


\subsection{Motivation}
In \cite{Shor16}, the author found a formula for the $q$-Weierstrass weight of branch points on a superelliptic curve.  In order to compute the weight, one needs to calculate the number of missing orders of vanishing in a certain basis of $q$-differentials as well as the sum of the missing orders.  These quantities are exactly $S_0(a,b;1)$ and $S_1(a,b;1)$.

In this paper, we are motivated by the problem of computing the higher-order Weierstrass weight of the point at infinity in a tower of curves defined by equations of superelliptic curves.  This follows work of Silverman, who investigated higher-order Weierstrass points on hyperelliptic curves in \cite{Silverman90}; of Towse, who looked at Weierstrass weights of branch points on superelliptic curves in \cite{Towse96}; and more recently of this paper's second author, who looked at higher-order Weierstrass weights of branch points on superelliptic curves in \cite{Shor16}.

In general, higher-order Weierstrass points are special points on an algebraic curve because their weights are invariant under automorphisms. One can use Weierstrass points to show a non-hyperelliptic curve of genus $g\geq2$ has a finite automorphism group. (See \cite{ShaskaShor2015Advances}, for example.) Mumford, in \cite{Mumford99}, has suggested that $q$-Weierstrass points are analogous to $q$-torsion points on an elliptic curve.

\subsection{Main results}
We generalize equation~\eqref{eqn:tuenter_result} for compound sequences (Theorem~\ref{th:tuenter-generalization}) and demonstrate a few applications.  First, we use power functions to get explicit formulas for $S_m(A,B)$ for $m=0,1,2,3$ (Proposition~\ref{prop:closed-forms}).  Second, we use an exponential function to generalize equation~\eqref{eqn:rodseth_result} to compute $S_m(A,B)$  (Theorem~\ref{thm:rodseth-generalization}).

With explicit formulas for the $m=0$ and $m=1$ cases, we obtain the following result (Theorem~\ref{thm:q-wt-p-infty-compound}).

\begin{thm*}
Let $A=(a_1,\dots,a_k), B=(b_1,\dots,b_k)\in\mathbb{N}^k$ with $\gcd(a_i,b_j)=1$ for all $i,j$.  For $S_0(A,B)$ as given in Proposition~\ref{prop:closed-forms}, suppose $S_0(A,B)\ge 2$.  For $1\leq i\leq k$, let $f_i(x)\in\mathbb{C}[x]$ be a separable polynomial of degree $b_i$.
Consider the affine curve 
\[A_k = \left\{(x_0,\dots,x_k)\in\mathbb{C}^{k+1} : x_i^{a_i}=f_i(x_{i-1})\text{ for } 1\leq i\leq k\right\}.\]  Assume the affine curve $A_k$ is nonsingular, and let $C_k$ be the nonsingular projective model of $A_k$.  (Examples of such curves are given in Section \ref{sec:tower-examples}.)  Then $C_k$ is a curve of genus $g=S_0(A,B)$ with one point at infinity, $P_\infty^k$, which has $q$-Weierstrass weight 
\begin{align*}
w^{(q)}(P_\infty^k) &= 
\begin{dcases*} 
\frac{S_0(A^2,B^2)}{12}-S_0(A,B) & for $q = 1$, \\
\frac{S_0(A^2,B^2)}{12} & for $q \ge 2$,
\end{dcases*}
\end{align*}
where $A^e$ and $B^e$ denote component-wise exponentiation.  In particular, given a particular curve $C_k$, the $q$-Weierstrass weight of the point at infinity is constant for all $q\geq2$.
\end{thm*}

This paper is organized as follows.  In Section~\ref{sec:semigroup-basics}, we review some background material on numerical semigroups generated by compound and geometric sequences.  In Section~\ref{sec:generalizing-section}, we prove a generalization of equation~\eqref{eqn:tuenter_result}. We use power functions to find explicit formulas for $S_m(A,B)$ for small $m$, and we look at special cases of geometric and supersymmetric sequences. We also consider the problem of ``non-nugget numbers,'' which appeared in an algebra textbook in the early 1990s.  We use an exponential function with our generalization of equation~\eqref{eqn:tuenter_result} to generalize the approach of \cite{Rodseth93}, resulting in a combinatorial formula for $S_m(A,B)$ involving Bernoulli numbers.  Subsequently, we transition to algebraic curves and Weierstrass points. We provide some background material on higher-order Weierstrass points in Section~\ref{sec:weierstrass-background}. In Section~\ref{sec:towers-calculations} we consider the problem of calculating the $q$-Weierstrass weight of points at infinity in towers of complex algebraic curves defined iteratively by equations of superelliptic curves, and we conclude with the description of a large family of towers of curves that satisfy the conditions of Theorem~\ref{thm:q-wt-p-infty-compound}.

\section{Numerical semigroups and compound sequences}\label{sec:semigroup-basics}

\subsection{Numerical semigroups}
In this section, we briefly describe some results on numerical semigroups. For a thorough treatment, see \cite{RosalesGarciaSanchez09}.

\begin{defn}
A \textit{numerical semigroup} $S$ is a non-empty subset of $\NZ$ which contains 0, is closed under addition, and has a finite complement in $\NZ$.  We denote the complement of $S$ in $\NZ$ by $H(S)$. 
The \emph{Frobenius number} of $S$, denoted $F(S)$, is the largest element of $\mathbb{Z}\setminus S$, and the \emph{genus} of $S$, denoted $g(S)$ is the cardinality of $H(S)$.\end{defn}
If $G$ is a non-empty subset of $\NZ$, let 
\[\langle G\rangle=\{n_0 g_0+\dots+ n_k g_k : k\in\NZ, n_i\in\NZ, g_i\in G\},\]
the submonoid of $\NZ$ generated by $G$.  It is well known that $\langle G\rangle$ is a numerical semigroup exactly when $\gcd(G)=1$.  (See \cite[Lemma 2.1]{RosalesGarciaSanchez09}.)  Furthermore, the complement of $\langle G\rangle$ is nonempty precisely when $1\not\in G$.

\begin{propn}[{\cite[Lemma 2.14, Proposition 4.4]{RosalesGarciaSanchez09}}]
For any numerical semigroup $S$ with nonempty complement, $2g(S)\geq F(S)+1$.  We have equality if and only if $F(S)$ is odd and $x\in H(S)$ implies $F(S)-x\in S$.
\end{propn}


If $2g(s)=F(S)+1$, we call $S$ a \emph{symmetric numerical semigroup}.  The symmetry comes from the property that, for any pair of non-negative integers that sum to $F(S)$, one integer is in $S$ and the other integer is in $H(S)$.


In this paper, we will only consider sets $S$ where $H(S)\neq\emptyset$.
\begin{cor}\label{cor:2g=F+1}
For $S$ a numerical semigroup, if $H(S)\neq\emptyset$, then $1\in H(S)$, so $F(S)-1\in S$.
\end{cor}


\subsection{Semigroups from compound sequences}\label{sec:semigroups}

For the benefit of the reader, we restate some definitions presented at the beginning of this paper.

\begin{defn}
For $k\in\mathbb{N}_0$, let $A=(a_1,\dots,a_k)$ and $B=(b_1,\dots,b_k)\in\mathbb{N}^k$. We say the pair $(A,B)$ is \textit{suitable} if $\gcd(a_i,b_j)=1$ for all $i\ge j$.  
Let $g_0=\prod_{i=1}^k a_i$ and, for $1\leq i\leq k$, let $g_i=g_{i-1}b_i/a_i$.  We say $(g_i)_{i=0}^k$ is a \textit{compound sequence}, and we denote it $G(A,B)$.  A set $G$ is \textit{compound} if its elements can be ordered to form a compound sequence, and, in an abuse of notation, we will write $G=G(A,B)$. Note that if $k=0$, then $A=B=()$ (the empty tuple), and one checks that $(A,B)$ is trivially a suitable pair so $G(A,B)=\{1\}$ is compound.
\end{defn}

Compound sequences are generalizations of  geometric sequences, which occur when $a_1=\dots=a_k$ and $b_1=\dots=b_k$.  Numerical semigroups arising from compound sequences have been studied in \cite{KiersONeillPonomarenko16}, following work on numerical semigroups from geometric sequences in \cite{OngPonomarenko08} and \cite{Tripathi08}.

\begin{remark}We note that our definition of compound sequence differs slightly from the definition given in \cite{KiersONeillPonomarenko16}, where there is the additional condition that $2\leq a_i<b_i$ for all $i$.  The following proposition is proved in their paper, though this additional condition isn't used in the proof so the result holds for our definition as well.
\end{remark}

\begin{propn}\cite[Proposition 2, property 5]{KiersONeillPonomarenko16}
If $G$ is compound, then $\gcd(G)=1$.  Thus, $\gcd(G(A,B))=1$ for any suitable pair $(A,B)$.
\end{propn}


If any $a_i$ or $b_i$ is 1, then there are consecutive terms in $G(A,B)$ where one divides the other, and thus $G(A,B)$ is not a minimal generating set of $\langle G(A,B)\rangle$.  Combined with \cite[Proposition 2, property 6]{KiersONeillPonomarenko16}, we have the following proposition.

\begin{propn}\label{prop:min-set-of-generators}
Let $(A,B)$ be a suitable pair. Then $G(A,B)$ is a minimal set of generators of $\langle G(A,B)\rangle$ if and only if $a_i,b_i \ge 2$ for all $i$.
\end{propn}

Furthermore, for any compound set $G$ we can find a suitable pair $(A,B)$ with $a_i,b_i \ge 2$ for all $i$ such that $G(A,B)$ minimally generates $\langle G\rangle$.

\begin{defn}
For any tuple $C=(c_1,\dots,c_k)$ and any $i=1,\dots,k$, let $\pi_i:\NN^k\to\NN^{k-1}$ be the projection which deletes the $i$th component of $C$, let $C^e$ denote component-wise exponentiation, and let $\rho(C)$ denote $C$ written in reverse order.  That is, $\pi_i(C):=(c_1,\dots,c_{i-1},c_{i+1},\dots,c_k)$, $C^e:=(c_1^e,\dots,c_k^e)$, and $\rho(C):=(c_k,\dots,c_1)$.
\end{defn}

\begin{propn}
If $G$ is a compound set, then $\langle G\rangle=\langle G(A,B)\rangle$ for some suitable pair $(A,B)$ with $a_i,b_i \ge 2$ for all $i$.
\end{propn}

\begin{proof}
If $1\in G$, then $\langle G\rangle=\mathbb{N}_0=\langle \{1\}\rangle=\langle G(A,B)\rangle$ for $0$-tuples $A=B=()$. Note that $(A,B)$ is a suitable pair, and it is vacuously true that $a_i,b_i\ge2$ for all $i$.

Now, we suppose $1\not\in G$.  Then we have $G=G(A,B)$ for some suitable pair $(A,B)$ of $k$-tuples with $1\not\in G(A,B)$ and $k\ge1$.
Let $A=(a_1,\dots,a_k)$ and $B=(b_1,\dots,b_k)$.  If $a_i,b_i \ge 2$ for all $i$, then we are done.  
Otherwise, we use the following procedure to eliminate any 1s in $A$ or $B$.

If $a_i=1$ for some $i$, then $g_i=b_i g_{i-1}$, so we can remove $g_i$ without affecting the semigroup generated by $G(A,B)$. To do so, let $A'=\pi_i(A)$ and $B'=(b_1,  \dots, b_{i-1}, b_i\cdot b_{i+1}, \dots, b_k)$. (In other words, replace $b_{i+1}$ with $b_i\cdot b_{i+1}$ and then delete $a_i$ and $b_i$.)  Then $G'=G(A',B')=(g_0,\dots,g_{i-1},g_{i+1},\dots,g_k)$, so $\langle G' \rangle = \langle G \rangle$, $\gcd(G')=1$, and $(A',B')$ is a suitable pair. It follows that $A'$ has one fewer 1 than $A$, and $B'$ has at most as many 1s as $B$.

If $b_i=1$ for some $i$, we can perform a similar process (replacing $a_{i-1}$ with $a_{i-1}\cdot a_i$ and then deleting $a_i$ and $b_i$) to produce a suitable pair $(A'',B'')$ such that, for $G''=G(A'',B'')$, we have $\langle G''\rangle=\langle G\rangle$ where $A''$ has at most as many 1s as $A$ and $B''$ has one fewer 1 than $B$.

We can iterate this procedure finitely many (at most $k-1$) times to produce a suitable pair $(\overline{A},\overline{B})$ such that $\overline{A}$ and $\overline{B}$ contain no 1s. (If one of our resulting tuples contained only 1s, then 1 would be in $G(A,B)$, which is a contradiction.)  Let $\overline{G}=G(\overline{A},\overline{B})$.  Then $\langle G\rangle$ is generated by $\overline{G}$ and all entries of $\overline{A}$ and $\overline{B}$ are greater than 1, as desired.
\end{proof}

To conclude, nothing is lost by only considering pairs $(A,B)$ where $a_i,b_i \ge 2$ for all $i$.  However, our results later in this paper are valid for the situations where we have some $a_i$ or $b_i$ equal to 1, so we allow that possibility.

The following proposition is straightforward.

\begin{propn} 
If $(A,B)$ is a suitable pair of $k$-tuples, then the following pairs of tuples are suitable for all $e_1,e_2\in\mathbb{N}$ and all $1 \leq i \le j \le k$:  $(A^{e_1}, B^{e_2})$; $(\rho(B), \rho(A))$; $(\pi_i(A), \pi_j(B))$.

Furthermore, $\rho(G(A,B))=G(\rho(B),\rho(A))$ as sequences, so $G(A,B)$ and $G(\rho(B),\rho(A))$ generate the same numerical semigroup.
\end{propn}

For the rest of this paper, we will  use the following notation.  For any $G\subseteq\mathbb{N}$, let $R(G):=\langle G\rangle$, the set of \emph{representable integers}, and let $\NR(G):=H(R(G))$, the set of \emph{non-representable integers}. If $\gcd(G)=1$, then $\NR(G)$ is a finite set, so for any $m\in\NZ$ we define the $m$th Sylvester sum \[S_m(G):=\sum\limits_{n\in \NR(G)}n^m.\]  
In what follows, it will be helpful to consider the sum of the generating elements, denoted $\sigma(G):=\sum\limits_{g\in G}g.$

Since we are interested in compound sequences, for a suitable pair $(A,B)$, we have the compound sequence $G=G(A,B)$ and we define $\NR(A,B):=\NR(G)$, $R(A,B):=R(G)$, $S_m(A,B):=S_m(G)$, and $\sigma(A,B):=\sigma(G)$.

We will also consider the special case of numerical semigroups generated by geometric sequences.  One can show, as in \cite[Section 1]{OngPonomarenko08}, that $G$ is a geometric sequence of natural numbers with $\gcd(G)=1$ if and only if $G=\{a^{k-i} b^{i} : 0\leq i\leq k\}$ and $\gcd(a,b)=1$.  Thus, a geometric sequence $G$ generates a numerical semigroup with finite complement exactly when $G=G(A,B)$ for some $A=(a,\dots,a)$ and $B=(b,\dots,b)$ with $\gcd(a,b)=1$; i.e. when $(A,B)$ is suitable. The complement is nonempty with the additional condition that $a,b>1$.  With such $k$-tuples $A$ and $B$, we will denote a geometric sequence by $G(a,b;k):=G(A,B)$.  Similarly, let $\NR(a,b;k):=\NR(A,B)$, $R(a,b;k):=R(A,B)$, $S_m(a,b;k):=S_m(A,B),$ and $\sigma(a,b;k):=\sigma(A,B).$  We note that 
\begin{align*}
\sigma(a,b;k)=\sigma(b,a;k)=\frac{a^{k+1}-b^{k+1}}{a-b}.
\end{align*}

Since $\rho(G(A,B))=G(\rho(B),\rho(A))$, we have $R(A,B)=R(\rho(B),\rho(A))$, so $S_m(A,B)=S_m(\rho(B),\rho(A))$ for all suitable pairs $(A,B)$.  In particular, $R_m(b,a;k)=R_m(a,b;k)$, so $S_m(a,b;k)$ is a symmetric function in $a$ and $b$ for all $m$. However, we do not have similar symmetry with $S_m(A,B)$ and $S_m(B,A)$ for $k>1$ because $G(A,B)$ and $G(B,A)$ are not necessarily equal as sets.  For instance, $G((a_1,a_2),(b_1,b_2))=\{a_1a_2,b_1a_2,b_1b_2\}$ and $G((b_1,b_2),(a_1,a_2))=\{b_1b_2,a_1b_2,a_1a_2\}$ are equal as sets if and only if they are geometric.


We will now cite some results on the genus and Frobenius number of numerical semigroups arising from compound sequences.  With these results, we are including the condition that $2\le a_i<b_i$ for all $i$.  In Section~\ref{sec:generalizing-section}, we will show that these results hold without this additional condition.

\begin{propn}[{\cite{KiersONeillPonomarenko16}}]\label{propn:KOP-compound-genus-value}
For any suitable pair $(A,B)$ with $2\leq a_i<b_i$ for all $i$, $R(A,B)$ is a symmetric numerical semigroup with Frobenius number $F(R(A,B)) = a_k b_1\, \sigma(\pi_k(A),\pi_1(B)) - \sigma(A,B)$ and genus $g(R(A,B)) = S_0(A,B)= \left(F(R(A,B)) + 1\right)/2.$
\end{propn}

\begin{proof}
That $R(A,B)$ is symmetric follows from \cite[Corollary 9]{KiersONeillPonomarenko16} and \cite[Corollary 9.12]{RosalesGarciaSanchez09}.
For the Frobenius number and genus, from \cite[Corollary 16]{KiersONeillPonomarenko16} we have \[F(R(A,B)) = -g_0 +\sum_{i=1}^k g_i(a_i-1).\] 
Thus, \begin{align*}
F(R(A,B)) 
&= -\sum\limits_{i=0}^k g_i + \sum\limits_{i=1}^k b_1\cdots b_i a_i\cdots a_k \\ 
& = -\sigma(A,B) + a_k b_1 \sum\limits_{i=1}^kb_2\cdots b_i a_i\cdots a_{k-1} \\ 
&= -\sigma(A,B) + a_k b_1\,\sigma(\pi_k(A),\pi_1(B)). \end{align*}

Since $R(A,B)$ is symmetric, as is noted in \cite[Corollary 17]{KiersONeillPonomarenko16}, $g(R(A,B))=(F(R(A,B))+1)/2$.
\end{proof}

\begin{cor}
For any suitable pair $(A,B)$ with $2\leq a_i<b_i$ for all $i$ and any $e_1,e_2\in\mathbb{N}$, 
\[F(R(A^{e_1},B^{e_2}))=a_k^{e_1} b_1^{e_2} \, \sigma(\pi_k(A^{e_1}),\pi_1(B^{e_2}))-\sigma(A^{e_1},B^{e_2})\]
and 
$g(R(A^{e_1},B^{e_2}))=S_0(A^{e_1},B^{e_2})=(F(R(A^{e_1},B^{e_2}))+1)/2$.
\end{cor}

\begin{cor}\label{cor:geometric-genus-value}
Suppose $G=G(a,b;k)$ is a compound set whose elements form a geometric sequence.  Then $R(a,b;k)$ is a symmetric numerical semigroup with \[F(R(a,b;k))=ab\,\sigma(a,b;k-1)-\sigma(a,b;k)\] and
\begin{align*}
g(R(a,b;k))=S_0(a,b;k)=\frac{ab\, \sigma(a,b;k-1)-\sigma(a,b;k)+1}{2}.
\end{align*}
\end{cor}

\begin{proof}
If $k=0$, or if $a$ or $b$ is 1, then $R(a,b;k)=\mathbb{N}_0$, so the Frobenius number is $-1$ and the genus is 0. The formulas give these values.

Now we assume $k\ge1$ and $a,b>1$.  Since $R(a,b;k)=R(b,a;k)$ with $\gcd(a,b)=1$, we may assume $a<b$.  The results then follow from Proposition \ref{propn:KOP-compound-genus-value}. 
\end{proof}

The result for the Frobenius number appears in \cite{OngPonomarenko08} and \cite[Theorem 1a]{Tripathi08}. 
The result for the genus appears in \cite[Theorem 1b]{Tripathi08}.  From \cite[Proposition 7]{Shor05} and \cite[Theorem 6]{Shor11}, we see that this can also be written as 
\begin{align}\label{eqn:S0_alt_formula}
S_0(a,b;k)=\frac{(b-1)a^{k+1}-(a-1)b^{k+1}+a-b}{2(a-b)}.
\end{align}
(These latter results are from work in positive characteristic, though they hold in characteristic zero as well.)


\section{Power sums from compound sequences}\label{sec:generalizing-section}

We will present a generalization of equation~\eqref{eqn:tuenter_result} for numerical subgroups arising from compound sequences and demonstrate some of its applications.

\subsection{A generalization of a theorem of Tuenter for compound sequences}
\label{sec:tuenter-generalization}


Our main tool is the following lemma which allows us to describe $\NR(A,B)$.

\begin{lemma} \label{lem:R/NR}
Let $(A,B)$ be a suitable pair of $k$-tuples, and let $G(A,B)=\{g_i : 0\leq i\leq k\}$.  For each integer $n \in \mathbb Z$ and each integer $0 \le j \le k$, there is a unique expression $n = \sum_{i=0}^k n_i g_i$, where $0 \le n_i < b_{i+1}$ for $0 \le i < j$ and $0 \le n_i < a_i$ for $j < i \le k$. Moreover, $n \in \NR(A,B)$ if and only if 
\begin{align*}
1 \le -n_j \le \frac{1}{g_j}\sum_{i \ne j} n_i g_i.
\end{align*}
\end{lemma}

\begin{proof}
Let $n \in \mathbb Z$. Since $(A,B)$ is suitable, $\gcd(G(A,B))=1$, so we can write $n = \sum_{i=0}^k n_ig_i$ for some $n_i\in\mathbb Z$. If any $n_i$ is outside the desired range, the ``excess" can be shifted towards $n_j$. For example, if it is not the case that $0 \le n_0 < b_1$, then writing $n_0 = qb_1+r$ where $0 \le r < b_1$, we have
\begin{align}\label{eq:right shift}
n_0g_0 = (qb_1+r)a_1\cdots a_k = rg_0 + (qa_1)g_1,
\end{align}
and the same goes for any $1 \le i < j$.  On the other end, if it is not the case that $0 \le n_k < a_k$, then write $n_k = qa_k + r$ where $0 \le r < a_k$. Then
\begin{align} \label{eq:left shift}
n_kg_k = (qa_k + r)b_1 \cdots b_k = (qb_k)g_{k-1} + rg_k,
\end{align}
and a similar shift fixes $n_i$ for $j < i < k$.

To show that this expression is unique, suppose $n = \sum_{i=0}^k n_i g_i = \sum_{i=0}^k m_i g_i$ where $0 \le n_i, m_i < b_{i+1}$ for $0 \le i < j$ and $0 \le n_i, m_i < a_{i}$ for $j < i \le k$. Then
\begin{align*}
0 = \sum_{i=0}^k (n_i - m_i)g_i,
\end{align*}
where $|n_i - m_i| < b_{i+1}$ for $0 \le i < j$ and $|n_i - m_i| < a_{i}$ for $j < i \le k$. Necessarily, $b_1 \mid (n_0 - m_0)$, and therefore $n_0 = m_0$. Consequently, $b_{2} \mid (n_1 - m_1)$, hence $n_1 = m_1$, and continuing this line of reasoning, $n_i = m_i$ for each $0 \le i < j$. Similarly from the other end, $a_k \mid (n_k - m_k)$. Thus $n_k = m_k$, which implies $a_{k-1} \mid (n_{k-1} - m_{k-1})$, and so on. Therefore as $n_i = m_i$ for $j < i \le k$, it follows that $n_j = m_j$ as well. 

If $n \in \NR(A, B)$, then writing $n = \sum_{i=0}^k n_i g_i$ with $0 \le n_i < b_{i+1}$ for $0 \le i < j$ and $0 \le n_i < a_{i}$ for $j < i \le k$ as above, it is necessary that $n_j < 0$. As $n > 0$, the bounds for $n_j$ follow immediately. 

For the converse, if $n\not\in \NR(A,B)$, then either $n\in R(A,B)$ or $n<0$.

If $n\in R(A,B)$, then we can write $n=\sum_{i=0}^k m_ig_i$ with $m_i\ge0$ for all $i$. Let $0\le j\le k$.  As above, if we do not have $m_i<b_{i+1}$ for all $i<j$ and $m_i<a_i$ for all $i>j$, we can use equations~\eqref{eq:right shift} and \eqref{eq:left shift} to shift the excess toward the coefficient of $g_j$ to get all other coefficients in the desired ranges. We obtain a new representation $n=\sum_{i=0}^k n_ig_i$ where $0\le n_i<b_{i+1}$ for $i<j$ and $0\le n_i<a_i$ for $i>j$. As a result of shifting, the coefficient of $g_j$ cannot decrease, so $n_j\ge m_j\ge 0$, as desired.

Finally, if $n<0$, then for any $j$, we have the unique expression $n=\sum_{i=0}^k n_ig_i<0$ with $0\le n_i<b_{i+1}$ for $i<j$ and $0\le n_i<a_i$ for $i>j$, so $-n_j>\sum_{i\ne j}n_ig_i/g_j$, as desired.
\end{proof}


The formula for the Frobenius number follows immediately.

\begin{cor}\label{cor:frob_compound}
Let $(A,B)$ be a suitable pair of $k$-tuples.  The Frobenius number of $R(A,B)$ is \[F(R(A,B))=a_kb_1\,\sigma(\pi_k(A),\pi_1(B))-\sigma(A,B).\]
\end{cor}

\begin{proof}
For any $j$ with $0\leq j\leq k$, let $F_j=\sum_{i=0}^k n_{i,j}g_i$ where \[n_{i,j} = \begin{dcases*}
b_{i+1}-1 & if $i<j$, \\ 
-1 & if $i=j$,\\
a_i-1 & if $i>j$.
\end{dcases*}\]  
That is, each $n_{i,j}$ is maximal in the sense of Lemma \ref{lem:R/NR}, so $F_j$ is maximal in $\NR(A,B)$. Therefore $F_j$ does not depend on $j$.  Indeed, 
\[F_j =-g_j+\sum_{i=0}^{j-1}(b_{i+1}-1)g_i+\sum_{i=j+1}^k(a_i-1)g_i=-\sigma(A,B)+\sum_{i=1}^kb_1\cdots b_ia_i\cdots a_k,\]
so $F(R(A,B))=a_kb_1\,\sigma(\pi_k(A),\pi_1(B))-\sigma(A,B)$.
\end{proof}

With Lemma \ref{lem:R/NR}, we can now generalize equation~\eqref{eqn:tuenter_result} for compound sequences.
Let 
\begin{align*}
&U_j = \left\{\sum_{i=0}^k n_ig_i : 0 \le n_i < b_{i+1} \text{ for } 0 \le i < j, \text{ and } 0 \le n_i < a_i \text{ for } j < i \le k\right\} \atext{and}\\
&U_{j,0} = \{n \in U_j : n_j = 0\}.
\end{align*}

\begin{thm} \label{th:tuenter-generalization}
For any suitable pair of $k$-tuples $(A,B)$ and any function $f$ defined on the non-negative integers,
\begin{align} \label{eq:tuenter}
\sum_{n \in \NR(A, B)} [f(n+g_j) - f(n)] = \sum_{n \in U_{j,0}} f(n) - \sum_{n=0}^{g_j-1}f(n).
\end{align}
\end{thm}

\begin{proof}
By Lemma \ref{lem:R/NR}, every integer is uniquely expressed as a value in $U_j$. Hence setting $V = \{n+g_j : n \in \NR(A, B)\}$, we have
\begin{align*}
&V \setminus \NR(A, B) = \{n \in U_{j,0} : n > g_j\}, \atext{and} \\
&\NR(A, B) \setminus V = \{n \in \NR(A, B) : n < g_j\}.
\end{align*}
The sum on the left of equation \eqref{eq:tuenter} is telescoping in the sense that whenever $n \in \NR(A, B)$ and $n+g_j \in \NR(A, B)$, these terms will be offset. Thus this sum may be written as a sum over disjoint sets,
\begin{align*}
\sum_{n \in \NR(A, B)} [f(n+g_j) - f(n)] = \sum_{n \in V^+} f(n) - \sum_{n \in V^-} f(n),
\end{align*}
where $V^+ = V \setminus \NR(A, B)$ and $V^- = \NR(A, B) \setminus V$. Noting that
\begin{align*}
&V^+ \cup \{n \in R(A, B) : 0 \le n < g_j\} = U_{j,0}, \atext{and} \\
&V^- \cup \{n \in R(A, B) : 0 \le n < g_j\} = \{0, \ldots, g_j-1\},
\end{align*}
it follows immediately that
\begin{align*}
\sum_{n \in V^+} f(n) - \sum_{n \in V^-} f(n) = \sum_{n \in U_{j,0}} f(n) - \sum_{n=0}^{g_j-1} f(n).
\end{align*}
\end{proof}

\begin{remark}
If one excludes the terms on the right of equation \eqref{eq:tuenter} where $n = 0$, then one obtains an identical result to Theorem \ref{th:tuenter-generalization} for functions defined on the positive integers.
\end{remark}

If $k=1$ and $j=0$, we recover equation~\eqref{eqn:tuenter_result}.

Finally, Theorem~\ref{th:tuenter-generalization} completely characterizes the set $\NR(A,B)$.  That is, all properties of $\NR(A,B)$ can be derived using this result.  The reasoning is the same as in the $k=1$ case (from \cite[Section 2]{Tuenter06}).

\subsection{Application with power functions}\label{sec:power-sums}
If we let $j=0$ in Theorem~\ref{th:tuenter-generalization}, we obtain the following identity which we will utilize in this section.
\begin{align}\label{eq:apply-tuenter}
\sum\limits_{n\in \NR(A,B)}(f(n+g_0)-f(n)) = \sum\limits_{n_1=0}^{a_1-1}\dots \sum\limits_{n_k=0}^{a_k-1}
f\left(\sum\limits_{i=1}^{k} n_i g_i \right) - \sum\limits_{n=0}^{g_0-1} f(n)
.\end{align}
We can use $f(n)=n^{m+1}$ to get explicit formulas for $S_m(A,B)$ for small values of $m$.

\begin{propn}\label{prop:closed-forms}
Let $(A,B)$ be a suitable pair of $k$-tuples.  Then
\begin{align*}
S_0(A,B) &= \frac{a_k b_1\,\sigma(\pi_k(A),\pi_1(B))-\sigma(A,B)+1}{2},\\
S_1(A,B) &= \frac{S_0(A,B)^2-S_0(A,B)}{2}+\frac{S_0(A^2,B^2)}{12},\\
S_2(A,B) &= \frac{(S_0(A,B)-1)S_0(A,B)(2S_0(A,B)-1)}{6}+\frac{S_0(A^2,B^2)}{12}\left(2S_0(A,B)-1\right), \\
S_3(A,B) &= \frac{(S_0(A,B)^2-S_0(A,B))^2}{4}+\frac{S_0(A^2,B^2)}{4}S_0(A,B)(S_0(A,B)-1)\\ & \quad + \frac{S_0(A^2,B^2)}{48}(S_0(A^2,B^2)+2)-\frac{S_0(A^4,B^4)}{240}.
\end{align*}
\end{propn}

\begin{proof}
The result for $m=0$ follows by substituting $f(n)=n$ into equation \eqref{eq:apply-tuenter} and is a straightforward calculation.  We will show the calculations for $m=1$.

Let $f(n)=n^2$. For notation let $\NR=\NR(A,B)$.  Then by equation~\eqref{eq:apply-tuenter}, 
\[\sum\limits_{n\in \NR}((n+g_0)^2-n^2) = \sum\limits_{n_1=0}^{a_1-1}\dots \sum\limits_{n_k=0}^{a_k-1}
\left(\sum\limits_{i=1}^{k} n_i  g_i\right)^2 - \sum_{n=0}^{g_0 - 1}n^2.\]
Simplifying both sides, we have
\begin{align}\label{eq:main-equation} g_0^2\,S_0(A,B)+2g_0S_1(A,B)
=  \sum\limits_{n_1=0}^{a_1-1}\dots \sum\limits_{n_k=0}^{a_k-1}
\left(\sum\limits_{i=1}^{k} n_i  g_i\right)\left(\sum\limits_{j=1}^{k} n_j g_j\right) - \frac{(g_0-1)g_0(2g_0-1)}{6}.\end{align}

Since there are a finite number of summations, each with a finite number of terms, we can change the order to evaluate the summations over $n_1,\dots,n_k$ first.  We will break them up into the cases where $i=j$ and $i\neq j$.  If $i=j$, we get \[\sum_{n_1=0}^{a_1-1}\dots \sum_{n_k=0}^{a_k-1} \sum_{i=1}^{k}n_i^2g_i^2 = g_0 \sum_{i=1}^{k}\frac{(a_{i}-1)(2a_{i}-1)g_i^2}{6}.\]  
If $i\neq j$, we get 
\begin{align*}
\sum_{n_1=0}^{a_1-1}\dots \sum_{n_k=0}^{a_k-1} \sum_{i=1}^{k}n_i g_i\sum_{j=1,j\neq i}^{k} n_j g_j
&= g_0 \sum_{i=1}^{k}\frac{(a_{i}-1)g_i}{2}\sum_{j=1,j\neq i}^{k}\frac{(a_{j}-1) g_j}{2}\\
&=\frac{g_0}{4}\left(\sum_{i=1}^{k}(a_{i}-1)g_i \right)^2 - \frac{g_0}{4}\sum_{i=1}^{k}((a_{i}-1)g_i)^2.
\end{align*}

Note that $\sum_{i=1}^{k}(a_{i}-1)g_i = 
2 S_0(A,B)+g_0-1.$ Similarly, $\sum_{i=1}^{k}(a_{i}^2-1)g_i^2 = 2 S_0(A^2,B^2)+g_0^2-1.$

The right side of equation~\eqref{eq:main-equation} simplifies to 
\[\frac{g_0}{12}(2 S_0(A^2,B^2)+g_0^2-1)+\frac{g_0}{4}(2S_0(A,B)+g_0-1)^2-\frac{(g_{0}-1)g_0(2g_0-1)}{6}.\]  We then subtract $g_0^2S_0(A,B)$ and divide through by $2g_0$ to obtain the following.
\[S_1(A,B)=\sum_{n\in \NR} n =\frac{S_0(A,B)^2-S_0(A,B)}{2}+\frac{S_0(A^2,B^2)}{12}.\]

Similar calculations give the results for $m=2,3$.  The work involves more cases to consider --- corresponding to combinations of indices which are equal or not equal to each other --- and to write down.  Since the ideas are the same, we omit the details.
\end{proof}


\begin{cor}
For any suitable pair of $k$-tuples $(A,B)$, $R(A,B)$ is a symmetric numerical semigroup.
\end{cor}
\begin{proof}
By Corollary~\ref{cor:frob_compound}, the Frobenius number of $R(A,B)$ is $F(R(A,B))=a_kb_1\,\sigma(\pi_k(A),\pi_1(B))-\sigma(A,B)$. By Proposition~\ref{prop:closed-forms}, the genus of $R(A,B)$ is $S_0(A,B)=(F(R(A,B))+1)/2$.  Thus, $R(A,B)$ is a symmetric numerical semigroup.
\end{proof}

For $m=0,1,2,3$, to calculate $S_m(A,B)$, one needs only to be able to compute $S_0(A^e,B^e)$ for $1\leq e\leq m+1$.

\begin{conj}
$S_m(A,B)$ is a polynomial in $S_0(A^e,B^e)$ (or, equivalently, in $\sigma(\pi_k(A^e),\pi_1(B^e))$ and $\sigma(A^e,B^e)$) for $1\leq e\leq m+1$.
\end{conj}

Also, we note that 
\begin{align*}
S_2(A,B)=(2S_0(A,B)-1)S_1(A,B)-\frac{(S_0(A,B)-1)S_0(A,B)(2S_0(A,B)-1)}{3}.
\end{align*}
That is, to compute the sum of squares of non-representable numbers, one need only know how many there are and what their sum is.  Written another way, we have 
\[S_2(A,B)=F(R(A,B))\cdot S_1(A,B)-2\sum_{n=0}^{S_0(A,B)-1}n^2.\] 

\subsubsection{Different generating sets with the same Sylvester sums}

As is mentioned at the end of Section~\ref{sec:tuenter-generalization}, $\NR(A,B)$ is completely characterized by Theorem~\ref{th:tuenter-generalization} by using $f(n)=n^m$ for $m$ sufficiently large.  However, we can still ask whether there exist different compound sets $G_1$ and $G_2$ such that $S_m(G_1)=S_m(G_2)$ for small values of $m$.

Suppose $k=1$.  Then $G_1=\{a,b\}$ and $G_2=\{c,d\}$ for $\gcd(a,b)=\gcd(c,d)=1$.  If $S_0(G_1)=S_0(G_2)$, then $(a-1)(b-1)/2 = (c-1)(d-1)/2$.  If $S_1(G_1)=S_1(G_2)$, then by Proposition~\ref{prop:closed-forms}, we have $S_0(G_1^2)=S_0(G_2^2)$, so $(a^2-1)(b^2-1)/2 = (c^2-1)(d^2-1)/2$.  Thus, if both $S_0(G_1)=S_0(G_2)$ and $S_1(G_1)=S_1(G_2)$ then $(a-1)(b-1)=(c-1)(d-1)$ and $(a+1)(b+1)=(c+1)(d+1)$. Thus, $ab=cd$ and $a+b=c+d$, which implies $\{a,b\}=\{c,d\}$, so $G_1=G_2$.

If $k>1$, the situation is different. Using Sage \cite{sagemathcloud} to search over the space of suitable pairs $(A,B)$ with $1<a_i,b_j<50$ and $k=2$, we find 124 pairs of unequal compound sets $G_1$ and $G_2$ such that $S_m(G_1)=S_m(G_2)$ for $m=0,1,2$ and $S_3(G_1)\neq S_3(G_2)$.  Ordering pairs by $S_0(G)$, the smallest example occurs when $G_1=\{16,10,35\}$ and $G_2=\{9,21,28\}$.  In this case, $G_1=G(A_1,B_1)$ for $A_1=(8,2)$ and $B_1=(5,7)$, and $G_2=G(A_2,B_2)$ for $A_2=(3,3)$ and $B_2=(7,4)$.  Then $S_0(G_1)=S_0(G_2)=45$, $S_1(G_1)=S_1(G_2)=1395$, $S_2(G_1)=S_2(G_2)=65415,$ and $S_3(G_1)=3746007\neq 3743235=S_3(G_2)$.

It would be interesting to find a method to determine different pairs of compound sets that have equal Sylvester sums for longer intervals of powers $m$.

\subsubsection{Applications to special types of sequences}

In the case $G$ is a geometric sequence of natural numbers, our formulas are further simplified.

\begin{cor}\label{cor:geom-power-sums}
If $G$ is a geometric sequence of natural numbers such that $\gcd(G)=1$, then $G=G(A,B)$ for $(A,B)$ a suitable pair of $k$-tuples where $A=(a,\dots,a)$, and $B=(b,\dots,b)$. Moreover, $\sigma(G)=\sigma(a,b;k)=\frac{a^{k+1}-b^{k+1}}{a-b}$ and $S_m(G)=S_m(a,b;k)$ for all $m$ where
\begin{align*}S_0(a,b;k)&= \frac{ab\,\sigma(a,b;k-1)-\sigma(a,b;k)+1}{2}, \\
S_1(a,b;k)&= \frac{S_0(a,b;k)^2-S_0(a,b;k)}{2}+\frac{S_0(a^2,b^2;k)}{12},\\
S_2(a,b;k)&= \frac{(S_0(a,b;k)-1)S_0(a,b;k)(2S_0(a,b;k)-1)}{6}+\frac{S_0(a^2,b^2;k)}{12}(2S_0(a,b;k)-1),\\
S_3(a,b;k)&=\frac{(S_0(a,b;k)^2-S_0(a,b;k))^2}{4}+\frac{S_0(a^2,b^2;k)}{4}S_0(a,b;k)(S_0(a,b;k)-1) \\ & \qquad + \frac{S_0(a^2,b^2;k)}{48}(S_0(a^2,b^2;k)+2)-\frac{S_0(a^4,b^4;k)}{240}.
\end{align*}
\end{cor}

Using equation~\eqref{eqn:S0_alt_formula}, we can simplify $S_1(a,b;k)$:
\begin{align*}
S_1(a,b;k) = \frac{1}{8}\left(\frac{(b-1)a^{k+1}-(a-1)b^{k+1}}{a-b} \right)^2 + \frac{1}{24}\left(\frac{(b^2-1)a^{2k+2}-(a^2-1)b^{2k+2}}{a^2-b^2} \right) -\frac{1}{12}.
\end{align*}
If $k=1$, we recover $S_1(a,b;1)=(a-1)(b-1)(2ab-a-b-1)/12$, as in equation~\eqref{eqn:brown_shiue_result}.

We can also consider compound sequences where the generators themselves have some symmetry.  For pairwise relatively prime integers $a_0,\dots,a_k$, let $g_i=\frac{a_0\cdots a_k}{a_i}$ and let $G=\{g_i:0\leq i\leq k\}$.  The numerical semigroup $R(G)$ is said to be \textit{supersymmetric}, and such semigroups were studied in \cite{HaggkvistFrobergGottlieb87}, where the authors showed that these semigroups are symmetric and computed the Frobenius number and genus.  Since these generating sets are compound, we can apply our methods to get the same results.

\begin{propn}
Let $a_0,\dots,a_{k}\in\mathbb{N}$ be pairwise relatively prime integers.  Let $A=(a_1,\dots,a_{k})$ and $B=(a_0,\dots,a_{k-1})$.  Then $(A,B)$ is a suitable pair, $g_i=\frac{a_0\cdots a_k}{a_i}$, and $G(A,B)=\left\{g_i : 0\leq i\leq k\right\}$, so 
\[S_0(A,B)=\frac{1}{2}\left(1+k\prod_{i=0}^k a_i-\sigma(A,B)\right)\]
and the Frobenius number of $R(A,B)$ is 
\[F(R(A,B))=k\prod_{i=0}^k a_i -\sigma(A,B).\]
\end{propn}


\subsubsection{Application to non-nugget numbers}
The following problem appears in \cite[Lesson 5.8]{WahPicciotto94}:
\begin{quote}
Eric tried to order 13 chicken nuggets at the fast food store. The employee informed him that he could order only 6, 9, or 20 nuggets. Eric realized he had to decide between ordering  $6 + 6 = 12$, or $6 + 9 = 15$.  What numbers of nuggets can be ordered by combining 6, 9, and 20? What numbers cannot be ordered? What is the greatest number that cannot be ordered? Explain. 
\end{quote}
While $(6,9,20)$ is not a compound sequence, we do get a compound sequence if we reorder its elements.
We let $A=(3,3)$ and $B=(2,10)$ and see that $(A,B)$ is a suitable pair with $G(A,B)=(3\cdot3,2\cdot3,2\cdot10)=(9,6,20)$.  This problem asks for a description of $\NR(A,B)$ and $F(R(A,B))$.  We will call elements of $\NR(A,B)$ ``non-nugget numbers.''

The cardinality of the set of non-nugget numbers is $S_0(A,B)=\frac{3\cdot 2(3+10)-(9+6+20)+1}{2}=22$, and the greatest number which cannot be ordered (i.e. the Frobenius number) is $F(R(A,B))=2S_0(A,B)-1=43$.

The sum of the non-nugget numbers, which is $S_1(A,B)$ can be calculated if we know $S_0(A,B)$ and $S_0(A^2,B^2)$.  Since $A^2=(3^2,3^2)$, $B^2=(2^2,10^2)$, and $G(A^2,B^2)=\{9^2,6^2,20^2\}$, we have $S_0(A^2,B^2)=\frac{3^2 2^3(3^2+10^2)-(9^2+6^2+20^2)+1}{2}=1704$.  Thus, $S_1(A,B) = \frac{22^2-22}{2}+\frac{1704}{12}=373$.

\subsection{Application with an exponential function}

We will now adopt the approach from \cite{Rodseth93} and use an exponential function to get a formula for $S_m(A,B)$.  For $f(n)=\exp(nz)$ (written this way to avoid subscripts in exponents) and $j=0$, again let $\NR=\NR(A,B)$ and consider the function 
\[h(z)=\sum_{n\in \NR}(\exp((n+g_0) z)-\exp(nz))=\left(\exp(g_0 z)-1\right)\sum\limits_{n\in \NR}\exp(nz).\]  
By equation~\eqref{eq:apply-tuenter}, 
\[h(z) = \sum\limits_{n_1=0}^{a_1-1}\dots \sum\limits_{n_{k}=0}^{a_k-1}
\exp\left(\sum\limits_{i=1}^{k} n_i g_i z \right) - \sum_{n=0}^{g_0-1}\exp(nz).\]  
Using finite geometric series, we have
\[h(z) = \prod_{i=1}^{k}\frac{\exp(g_i a_i z)-1}{\exp(g_i z)-1} - \frac{\exp(g_0 z)-1}{\exp(z)-1}.\] 
Multiplying both sides by $z/(\exp(g_0 z)-1)$, we find 
\[\sum_{n\in \NR}z\exp(nz) =  \prod_{i=0}^{k}\frac{z}{\exp(g_i z)-1}\prod_{j=1}^{k}\frac{\exp(g_j a_j z)-1}{z} - \frac{z}{\exp(z)-1}.\]
Using Taylor expansions, we get
\[\sum_{n\in \NR} \sum_{m=0}^\infty n^m \frac{z^{m+1}}{m!} = \prod_{i=0}^k\left(\sum_{x_i=0}^\infty B_{x_i}g_i^{x_i-1}\frac{z^{x_i}}{x_i!}\right)\prod_{j=1}^{k}\left(\sum_{y_j=0}^\infty (g_j a_j)^{y_j+1}\frac{z^{y_j}}{(y_j+1)!}\right) - \sum_{m=0}^\infty B_m\frac{z^m}{m!},\] where $B_0,B_1,B_2,\dots$ are Bernoulli numbers.
We equate coefficients of $z^m$ to obtain
\[\sum_{n\in \NR}\frac{n^{m-1}}{(m-1)!} = \sum_{\sum x_i + \sum y_j = m} \prod_{i=0}^k \left(\frac{B_{x_i}}{x_i!}\right) \prod_{j=1}^{k} \left(\frac{g_i^{x_i-1}(g_j a_j)^{y_j+1}}{(y_j+1)!}\right) -\frac{B_m}{m!},\] where the summation is over non-negative $x_i,y_j$ such that $\sum_{i=0}^k x_i + \sum_{j=1}^{k}y_j = m$. From this we have our result.
\begin{thm}\label{thm:rodseth-generalization}For a suitable pair $(A,B)$ of  $k$-tuples, 
\[S_{m-1}(A,B) = \frac{m!}{m}\sum_{\sum x_i+\sum y_j=m}\left(\prod_{i=0}^k\frac{B_{x_i}}{x_i!}\right) \frac{\prod_{i=1}^k a_i^{\alpha(i)} b_i^{\beta(i)}} {\prod_{j=1}^{k}(y_j+1)!}-\frac{B_m}{m},\]
where 
$\alpha(i):=\sum_{\ell=1}^{i}(x_{\ell-1}+y_{\ell})$ and
$\beta(i):=\sum_{\ell=i}^k (x_\ell+y_{\ell})$.
\end{thm}

We have the special case where $G(A,B)$ is a geometric sequence.
\begin{cor}
If $G=G(a,b;k)$ is a geometric sequence, then 
\[S_{m-1}(a,b;k) = \frac{m!}{m}\sum_{\sum x_i + \sum y_j = m}\left(\prod_{i=0}^k\frac{B_{x_i}}{x_i!}\right) \frac{a^{\alpha}b^{\beta}}{\prod_{j=1}^{k}(y_j+1)!}-\frac{B_m}{m}, \]
where 
$\alpha=\sum_{\ell=1}^{k}(k+1-\ell)(x_{\ell-1}+y_\ell)$ and 
$\beta=\sum_{\ell=1}^k \ell(x_\ell+y_{\ell})$.
\end{cor}

\begin{cor}
$S_m(A,B)$ is a polynomial in $a_i,b_j$ of total degree $(m+1)(k+1)$.
\end{cor}
\begin{proof}
Ignoring coefficients, each monomial in the summation formula for $S_{m-1}(A,B)$ is of the form 
\[ M = \prod_{i=1}^k a_i^{\sum\limits_{\ell=1}^{i} (x_{\ell-1}+y_\ell)} b_i^{\sum\limits_{\ell=i}^{k}(x_\ell+y_{\ell})}\]
for some non-negative integers $x_0,\dots,x_k,y_1,\dots,y_k$ such that $x_0+\dots+x_k+y_1+\dots+y_k=m$.  The degree of $M$ is $\sum_{\ell=1}^{k}(k+1-\ell)(x_{\ell-1}+y_\ell) + \sum_{\ell=1}^k\ell(x_\ell+y_{\ell}) = km+\sum_{\ell=1}^{k}y_\ell,$ so \[km\leq \deg(M)\leq (k+1)m.\]
The degree of $M$ is $(k+1)m$ precisely when $\sum y_\ell=m$ or, equivalently, $x_0=\dots=x_k=0$.  In that case, the Bernoulli numbers in the coefficient of $M$ are all 1, so the coefficient of $M$ is positive and thus non-zero. Hence $\deg(S_{m-1}(A,B))=(k+1)m$, and therefore $\deg(S_m(A,B))=(k+1)(m+1)$.
\end{proof}

\section{Higher-order Weierstrass points on algebraic curves}\label{sec:weierstrass-background}
\subsection{Divisors associated to functions and differentials}
We follow the background material and notation on Weierstrass points from \cite[Section 2]{ShaskaShor15}.  We will include the major results here along with specifics on calculating divisors associated to functions and differentials. 

Let $K$ be an algebraically closed field.  Let $C$ be a non-singular projective curve over $K$ of genus $g$, $K(C)$ its function field, and $K(C)^\times$ the invertible elements of $K(C)$.  Let $P$ denote an arbitrary $K$-rational point on $C$.  A divisor $D$ on $C$ is a formal sum 
\[D=\sum\limits_{P\in C} n_P P\]
for $n_P \in \mathbb{Z}$ with almost all $n_P=0$.  The set of divisors on a curve along with addition forms an abelian group.  We say a divisor $D$ is effective if $n_P\geq0$ for all $P$, and the degree of $D$ is $\deg(D)=\sum n_P$.  For any point $P$, let $\nu_P(D)=n_P$.

Suppose $f\in K(C)^\times$ with Laurent series $f(t)=\sum_{i=N}^\infty a_i t^i$ with $a_N\neq0$.  Let $\ord_{t=0}(f(t)):=N$.  
Suppose we have a point $P=(p_1,\dots,p_n)$ in affine coordinates $x_1,\dots,x_n$.  We can parametrize $C$ at $P$ with power series $x_1(t),\dots,x_n(t)$ in a local parameter $t$ such that $(x_1(0),\dots,x_n(0))=P$.  At $P$, for any function $f\in K(C)^\times$, we can write $f=f(x_1(t),\dots,x_n(t))$.  The order of vanishing of $f$ at $P$ is $\ord_P(f):=\ord_{t=0}(f(t)).$  We can then define the divisor $\dv(f)$ of a function $f$ as 
\[\dv(f) := \sum\limits_{P\in C}\ord_P(f)P.\]
The zero and pole divisors of $f$ are, respectively, $\dv(f)_0$ and $\dv(f)_\infty$, defined by 
\[\dv(f)_0:=\sum\limits_{f(P)=0}\nu_P(f)P, \quad \text{ and } \quad \dv(f)_\infty:=-\sum\limits_{f(P)=\infty} \nu_P(f)P.\]
Note that $\dv(f)_0$ and $\dv(f)_\infty$ are effective divisors of finite degree and that $\dv(f)=\dv(f)_0-\dv(f)_\infty$.  It also happens that $\deg(\dv(f))=0$.  That is, any $f\in K(C)^\times$ has as many zeros as it has poles, counting multiplicity.

We can also associate a divisor to any nonzero differential form $\omega$ on $C$.  At any point $P$, in terms of a local coordinate $t$ we can write $\omega=h\,\mathrm{d}t$ for $h\in K(C)^\times$.  Then $\dv(\omega)=\sum_{P\in C}\nu_P(h)P.$  In particular, for any function $f\in K(C)\setminus K$ and any point $P$, we have a local parametrization $f(t)$, so $\mathrm{d}f=f'(t)\,\mathrm{d}t$, and so \[\dv(\mathrm{d}f) = \sum\limits_{P\in C}\nu_P(f'(t))P.\]

\begin{hide}
For any divisor $D$ on $C$, let $\mathcal{L}(D)=\{ f\in k(C) : \dv (f)+D\geq0\}\cup\{0\}$ and $\ell(D)= \dim_k(\mathcal{L}(D))$.  By the Riemann-Roch theorem, for any canonical divisor $K=\dv(\omega)$, we have 
\[ \ell(D)-\ell(K-D)=\deg(D)+1-g.\]
Since the degree of the canonical divisor is $2g-2$, and since $\mathcal{L}(D)=\{0\}$ for any divisor $D$ with negative degree, if $\deg(D)\geq2g-1$, then $\deg(K-D)<0$, so $\ell(K-D)=0$.  Thus, if $\deg(D)\geq 2g-1$, then 
\[ \ell(D)=\deg(D)+1-g.\]
Let $P$ be a degree 1 point on $C$.  Consider the chain of vector spaces \[\mathcal{L}(0)\subseteq\mathcal{L}(P)\subseteq\mathcal{L}(2P)\subseteq\mathcal{L}(3P)\subseteq \dots \subseteq\mathcal{L}\left((2g-1)P\right).\]  Since $\mathcal{L}(0)=k$, we have $\ell(0)=1$.  And $\ell\left((2g-1)P\right)=g$.   We obtain the corresponding non-decreasing sequence of integers $$ \ell(0)=1, \ell(P), \ell(2P), \ell(3P), \dots, \ell\left((2g-1)P\right)=g.$$  It is straightforward to show that $0\leq \ell(nP)-\ell((n-1)P)\leq 1$ for all $n\in\mathbb{N}$.  If $\ell(nP)=\ell((n-1)P)$, then we call $n$ a \textit{Weierstrass gap number}.  For any point $P$, there are exactly $g$ Weierstrass gap numbers.  If the gap numbers are $1,2,\dots,g$, then $P$ is an \textit{ordinary point}.  Otherwise, we call $P$ a \textit{Weierstrass point}.  (Equivalently, we call $P$ a Weierstrass point if $\ell(gP)>1$.)

Note that if $C$ has genus $g=0$ then any point $P$ has no gap numbers and is therefore not a Weierstrass point.  If $C$ has genus $g=1$ then for any point $P$, $\ell(0)=\ell(P)=1$, so 1 is the only gap number of $P$, so $P$ is not a Weierstrass point.  Thus, if $g<2$, then $C$ has no Weierstrass points.  As we will see in Corollary~\ref{cor:w-pts-exist} below, if $g\geq2$, then $C$ has Weierstrass points as well as higher-order Weierstrass points which we describe below.

\subsection{Higher-order Weierstrass points}
Using differentials, we can define higher-order Weierstrass points (or $q$-Weierstrass points) as in \cite[Chapter III.5]{FarkasKra92}.
\end{hide}

\subsection{Higher-order Weierstrass points}
For any $q\in\mathbb{N}$, we now consider $q$-differentials. For the rest of this paper, because curves of genus $g\le 1$ do not contain any higher-order Weierstrass points, we will assume that $C$ is an algebraic curve of genus $g\geq2$.  Let $H^0(C,(\Omega^1)^q)$ be the $K$-vector space of holomorphic $q$-differentials on $C$, a vector space of dimension $d_q$. By Riemann-Roch, 
\begin{equation*}\label{d_q}
d_q = 
\begin{dcases*}
g & if $q=1$, \\ (g-1)(2q-1) & if $q\geq2$.
\end{dcases*}
\end{equation*}

For $P$ a $K$-rational point on $C$, there exists a basis $\{\psi_1,\dots,\psi_{d_q}\}$ of $H^0(C,(\Omega^1)^q)$ such that $ \ord_P(\psi_1)<\ord_P(\psi_2)<\cdots<\ord_P(\psi_{d_q}).$
The \textit{$q$-Weierstrass weight} (or \textit{$q$-weight}) of $P$ is 
\[w^{(q)}(P)=\sum_{i=1}^{d_q} \ord_P(\psi_i) - \sum_{j=0}^{d_q-1} j.\] We call the point $P$ a \textit{$q$-Weierstrass point} (or \textit{higher-order Weierstrass point}) if $w^{(q)}(P)>0$.  For any curve $C$ of genus $g\geq2$ and any fixed $q$, there are finitely many $q$-Weierstrass points for each $q\geq1$.  

\section{Calculations for branch points in towers}\label{sec:towers-calculations}
We will use the results from the previous section to compute the $q$-Weierstrass weight for the point at infinity on a tower of curves coming from defining equations of superelliptic curves.  In this section, we are working over $K=\mathbb{C}$.

Some of our work follows from \cite{Shor11} and \cite{Shor05}.  Those papers considered towers arising from equations of $C_{ab}$ curves, which were first described in \cite{Miura93}.  For relatively prime $a,b\in\mathbb{N}$, a $C_{ab}$ curve is a curve given by the affine equation 
\begin{align*}
\alpha_{a,0}x^a + \alpha_{0,b}y^b + \sum\limits_{i,j}\alpha_{i,j} x^i y^j=0,
\end{align*}
for constants $\alpha_{i,j}$ such that $\alpha_{a,0}\alpha_{0,b}\neq0$, and where the summation is over non-negative integers $i,j$ such that $aj+bi<ab$.  These curves are nice to work with because they have a single point at infinity $P_\infty$ and the functions $x$ and $y$ have poles of orders $b$ and $a$ (respectively) at $P_\infty$.

The above-mentioned papers were motivated by questions in coding theory, so they were set in fields of positive characteristic.  Fortunately, the results also hold for fields of characteristic zero so we can use them here.

Superelliptic curves, which we will describe below, are special cases of $C_{ab}$ curves.  We use superelliptic curves here rather than more general $C_{ab}$ curves because the ramification is easier to control.  In Section~\ref{sec:tower-description}, we describe the towers of curves and some of their properties.  In Section~\ref{sec:tower-calculations}, we compute a basis of holomorphic $q$-differentials and use that to compute the $q$-Weierstrass weight of the point at infinity.  In Section~\ref{sec:tower-examples}, we give examples of families of suitable towers.

\subsection{Tower description and divisors}\label{sec:tower-description}

For $k\in\mathbb{N}_0$, let $A=(a_1,\dots,a_k), B=(b_1,\dots,b_k)\in\mathbb{N}^k$ be  such that 
$\gcd(a_i,b_j)=1$ for all $i,j$.  (Note that the pair $(A,B)$ is suitable, with an additional gcd restriction.)  For $i=1,\dots,k$, let $H_i(x_{i-1},x_{i})=x_{i}^{a_i}-f_{i}(x_{i-1})$ where $f_i(x)\in\mathbb{C}[x]$ is a separable polynomial of degree $b_i$.  A plane curve defined by the single equation $H_i(x_{i-1},x_i)=0$ is called a \textit{superelliptic curve} when $a_i\geq2$.

Consider the algebraic curve $A_k$ given in affine coordinates by 
\[ A_k = \left\{ (x_0,x_1,\dots,x_k)\in\mathbb{C}^{k+1} : H_i(x_{i-1},x_i)=0\text{ for } 1\leq i\leq k\right\}.\] For the rest of this section, we will only consider curves $A_k$ that have no singular affine points.  (We provide examples of such curves in Section~\ref{sec:tower-examples}.)  Let $C_k$ be the desingularization of the projective curve $A_k$.  Since curves of genus $g\le 1$ have no higher-order Weierstrass points, we will further assume that $g(C_k)\geq2$.

We obtain a tower of curves 
\[\cdots \xrightarrow{\pi_{k+1}} C_{k} \xrightarrow{\pi_k} C_{k-1} \xrightarrow{\pi_{k-1}} \dots \rightarrow C_1 \xrightarrow{\pi_1} C_0 = \mathbb{P}^1(\mathbb{C})\] where $\pi_j(x_0,\dots,x_{j-1},x_j)=(x_0,\dots,x_{j-1}).$

As in Section~\ref{sec:semigroups}, let $g_0=\prod_{j=1}^k a_j$ and, for $1\leq i\leq k$, let $g_i=g_{i-1}b_i/a_i$. Then $G(A,B)=(g_i)_{i=0}^k$ is a compound sequence.

\begin{propn}
The curve $C_k$ has a single point at infinity, denoted $P_{\infty}^k$ which is totally ramified throughout the tower.  For $\nu_{\infty,k}(f)$, the valuation of a function $f$ at $P_{\infty}^k$, we have $\nu_{\infty,k}(x_i)=g_i$ for each $i=0,\dots,k$.
\end{propn}

\begin{proof}
For each $i=0,\dots,k$, let $F_i$ be the function field associated to the curve $C_i$.  That is, let 
\[F_i=\mathbb{C}(x_0,\dots,x_i)/\langle H_1(x_0,x_1),\dots,H_i(x_{i-1},x_i)\rangle.\]
Since $x_{i}^{a_i}=f_i(x_{i-1})$, $F_i/F_{i-1}$ is an algebraic extension of degree $[F_i:F_{i-1}]=a_i$.  Let $P_{\infty,0}$ be the place at infinity in $F_0$ and let $P_{\infty,i}$ be a place of $F_i$ lying over $P_{\infty,0}$ with associated valuation $\nu_{\infty,i}$.  (A place is a maximal ideal of a local ring of a function field. Places are in one-to-one correspondence with points on the associated curve, and valuations of functions at places and their corresponding points are equal.  For a reference, see \cite{Stichtenoth09}.)

For any $j=0,\dots,i$, we have $\nu_{\infty,i}(x_j^{a_j})=\nu_{\infty,i}(f_j(x_{j-1}))$, so $a_j\, \nu_{\infty,i}(x_j) = b_j\,\nu_{\infty,i}(x_{j-1})$.  Thus 
\[\nu_{\infty,i}(x_i) = \frac{b_i\cdots b_1}{a_i\cdots a_1}\nu_{\infty,i}(x_0).\] Since $\gcd(a_i,b_j)=1$ for all $i,j$, we conclude that $a_1\cdots a_i \mid \nu_{\infty,i}(x_0)$.\footnote{This is why our define our curves in this section with $\gcd(a_i,b_j)=1$ for all $i,j$ rather than just for all $i\geq j$.}  In particular, $g_0=a_1\cdots a_k \mid \nu_{\infty,k}(x_0)$.  Since $F_k/F_0$ is an extension of degree $a_1\cdots a_k=g_0$, we must have $\nu_{\infty,k}(x_0)\leq g_0$.  Thus, $\nu_{\infty,k}(x_0)=g_0$, which means $P_{\infty,k}$ is the unique place of $F_k$ lying over $P_{\infty,0}$.  Since there is a unique place at infinity, there is a unique point at infinity on $C_k$ which we denote $P_\infty^k$, and the valuation of a function $f$ at $P_\infty^k$ is $\nu_{\infty,k}(f)$.

Finally, since $\nu_{\infty,k}(x_i)=\frac{b_i}{a_i}\nu_{\infty,k}(x_{i-1})$ and $\nu_{\infty,k}(x_0)=g_0$, it follows that $\nu_{\infty,k}(x_i)=g_i$.
\end{proof}

Next, we will show the genus of $C_k$ is $S_0(A,B)$.  To do so, we will consider the Riemann-Roch space $\mathcal{L}(nP_\infty^k)$ for $n$ large and determine the number of missing pole orders, which is the genus by the Riemann-Roch Theorem.

We will modify the argument from \cite[Section 3]{Shor11}.  Let \[I_k=\langle H_1(x_{0},x_1), \dots, H_k(x_{k-1},x_k)\rangle\] be the ideal of the curve $C_k$, and consider the polynomial ring $\Gamma_k=\mathbb{C}[x_0,\dots,x_k]/I_k$.  We say a monomial $f\in\Gamma_k$ is \textit{$A$-reduced} if $f=x_0^{e_0}x_1^{e_1}x_2^{e_2}\dots x_k^{e_k}$ where $0\leq e_0$ and $0\leq e_i<a_i$ for $1\leq i\leq k$.  

\begin{lemma}
If $h\in\Gamma_k$, then $h$ can be written as a linear combination of $A$-reduced monomials.\end{lemma}
\proof
Since $x_i^{a_i}=f_i(x_{i-1})$, we can reduce any powers of $x_i$ to be at most $a_i$.  This will not affect the powers of $x_j$ for $j>i$.  Thus, we first reduce powers of $x_k$, then $x_{k-1}$, and so on to $x_1$.  Note that the powers of $x_0$ may be arbitrarily large.
\qed

\begin{lemma}
If $h_1,h_2$ are two $A$-reduced monomials, then $\nu_{\infty,k}(h_1)=\nu_{\infty,k}(h_2)$ if and only if $h_1=h_2$.
\end{lemma}
\proof
Clearly, if $h_1=h_2$, then $\nu_{\infty,k}(h_1)=\nu_{\infty,k}(h_2)$.

Now, suppose $\nu_{\infty,k}(h_1)=\nu_{\infty,k}(h_2)$ for $h_1=\prod_{i=0}^k x_i^{e_{i,1}}$ and $h_2=\prod_{i=0}^k x_i^{e_{i,2}}$.  Then \[\sum_{i=0}^k e_{i,1}g_i = \sum_{i=0}^k e_{i,2}g_i.\] Applying Lemma~\ref{lem:R/NR} with $j=0$, it follows that $e_{i,1}=e_{i,2}$ for $0\leq i\leq k$.  Thus, $h_1=h_2$.
\qed

\begin{propn}
Let $S$ be the semigroup of pole orders at $P_\infty^k$ in $F_k$ generated by elements of $\Gamma_k$.  Suppose, for some $r>0$ and $r\not\in S$, that there exists $\psi\in F_k$ such with $\dv(\psi)_\infty=rP_\infty^k$.  Then, for $\psi=f/h$ with $f,h\in\Gamma$, there is a place in the support of $\dv(h)_0$ corresponding to a singular point $Q$ on $C_k$.
\end{propn}
\proof
This is proved in \cite[Theorem 4]{Shor11} for a tower defined recursively by one polynomial (i.e. the situation where $H_1=H_2=\dots=H_k$). However, in that setting the fact that one has the same polynomial in each level of the tower is not needed for the proof, so the proof holds for our situation as well.
\qed

Thus, if the affine curve $A_k$ is nonsingular, we have \[\left\{\nu_{\infty,k}(\psi) : \psi\in F_k, \dv(\psi)_\infty = nP_\infty^k \text{ for some }n\in\mathbb{N}_0 \right\} = \left\{\nu_{\infty,k}(h) : h\text{ is an $A$-reduced monomial} \right\}. \]  That is, to determine a basis for $\mathcal{L}(nP_\infty^k)$, rather than considering all rational functions, we need only consider $A$-reduced monomials.

\begin{propn}
The genus of $C_k$ is $g(C_k)=S_0(A,B)$.
\end{propn}
\proof
By Riemann-Roch, $\dim\mathcal{L}(nP_\infty^k)=n+1-g$ for $n\geq2g-1$.  In other words, the genus of $C_k$ is the number of non-negative integers which are not the pole order of any $A$-reduced monomial.  Since the pole order of the monomial $\prod_{i=0}^k x_i^{e_i}$ is $\sum_{i=0}^k e_i g_i$, we see that $g(C_k)$ is the cardinality of the set of non-representable integers $\NR(A,B)$, which is $S_0(A,B)$.
\qed

\begin{lemma}\label{lem:coords_not_zero}
Let $P=(p_0,\dots,p_k)\in C_k$ be an affine point.  If $f_i'(p_i)=0$ for some $0\leq i< k$, then $p_i\neq0$ for all $0< i\leq k$.
\end{lemma}
\proof
In this case, the rank of the Jacobian matrix would drop, and so $P$ would be a singular point. This is a contradiction to the assumption that the affine curve $A_k$ is nonsingular.
\qed

In order to find a basis for the space of holomorphic $q$-differentials, we will consider the differential \[\omega=\frac{\mathrm{d}x_0}{\prod_{i=1}^k x_i^{a_i-1}}.\]

\begin{propn}\label{prop:no-affine-points}
Let $P\in C_k$ be an affine point.  Then $\nu_P\left(\omega\right)=0.$
\end{propn}

\begin{proof}
Suppose $P=(p_0,\dots,p_k)$ is a nonsingular affine point.  
For this proof, we will calculate $\nu_P(\mathrm{d}x_0)$ and $\nu_P(x_i)$ for $1\leq i\leq k$.

We begin with a parameterization of $C_k$ at $P$ given by 
\[x_i(t)=p_i+\sum\limits_{j=1}^\infty {c_{i,j}t^j}\]
for $0\leq i\leq k$ and some $c_{i,j}\in\mathbb{C}$.  Since $P$ is nonsingular, we have at least one $c_{i,1}\neq0$ for some $i$.

For each $i$ we have $x_{i}^{a_i} = f_{i}(x_{i-1}).$  Since $\deg(f_i)=b_i$, $f_i$ equals its Taylor polynomial of degree $b_i$ centered at $x_{i-1}=p_{i-1}$.  We have $f_{i}(x_{i-1})=\sum_{\ell=0}^{b_i}\frac{f^{(\ell)}(p_{i-1})}{\ell!}(x_{i-1}(t)-p_{i-1})^\ell$, so 
\begin{align}\label{eqn:coefficients}
\left(p_{i}+\sum\limits_{j=1}^\infty c_{i,j}t^j\right)^{a_{i}} = \sum_{\ell=0}^{b_i} \frac{f_i^{(\ell)}(p_{i-1})}{\ell!}\left(\sum_{j=1}^\infty c_{i-1,j}t^j\right)^\ell.
\end{align}
Equating coefficients of $1,t,t^2,\dots$, we have
\begin{align*}
p_{i}^{a_i} & = f_i(p_{i-1}), \\
\binom{a_i}{1} p_{i}^{a_i-1}c_{i,1} & = f_i'(p_{i-1})c_{{i-1},1}, \\
\binom{a_i}{1} p_{i}^{a_i-1}c_{i,2}+\binom{a_i}{2}p_{i}^{a_i-2}c_{i,1}^2 & = f_i'(p_{i-1})c_{i-1,2}+\frac{f_i''(p_{i-1})}{2!}c_{i-1,1}^2,
\end{align*}
and so on.

By Lemma \ref{lem:coords_not_zero} if $f'(p_j)=0$ for some $j<k$ then $p_i\neq0$ for all $i>0$. Equivalently, if $p_{i}=0$ for some $i>0$, then $f'(p_j)\neq0$ for all $j<k$.  We proceed by considering the three possibilities for each $i$: where $p_{i}\neq0$ and $f_i'(p_{i-1})\neq 0$; where $p_{i}=0$; or where $f_i'(p_{i-1})=0$.

Case 1. Suppose $p_{i}\neq0$ and $f_i'(p_{i-1})\neq0$.  
Suppose $\ord_P(x_{i-1}-p_{i-1})=m\geq1$. Then $c_{i-1,j}=0$ for $1\leq j<m$ and $c_{i-1,m}\neq 0$. Equating coefficients of $t^1,t^2,\dots,t^{m-1}$, we find $c_{i,j}=0$ for $1\leq j<m$. Equating coefficients of $t^{m}$, we have \[c_{i,m} = \frac{f_i'(p_{i-1})c_{i-1,m}}{a_ip_{i}^{a_i-1}}\neq0,\] so $\ord_P(x_{i}-p_{i})=m=\ord_P(x_{i-1}-p_{i-1})$.

Case 2. Suppose $p_{i}=0$ and $f_i'(p_{i-1})\neq0$.  Since $p_{i}=0$, the coefficients of $t,t^2,\dots,t^{a_i-1}$ on the left side of equation~\eqref{eqn:coefficients} are all zero, so $c_{i-1,1}=\dots=c_{i-1,a_i-1}=0$.  If $\ord_P(x_{i})=m$, 
then $c_{i,1}=\dots=c_{i,m-1}=0$ and $c_{i,m}\neq0$, so the coefficients of $t,t^2,\dots,t^{ma_i-1}$ are all zero, implying $c_{i-1,1}=\dots=c_{i-1,ma_i-1}=0$.  Equating coefficients of $t^{ma_i}$, we see $c_{i-1,ma_i}=c_{i,m}^{a_i}/f_i'(p_{i-1})\neq0$.  Thus, $\ord_P(x_{i-1}-p_{i-1})=ma_i=a_i\cdot\ord_P(x_{i}-p_{i})$. 

Case 3. Suppose $f_i'(p_{i-1})=0$ and $p_{i}\neq0$.  For some $r$ with $1\leq r<b_i$, we may assume $f_i'(p_{i-1})=f_i''(p_{i-1})=\dots=f_i^{(r)}(p_{i-1})=0$  and $f_i^{(r+1)}(p_{i-1})\neq0$.  Suppose $\ord_P(x_{i-1}-p_{i-1})=m\geq1$. 
Then $c_{i-1,j}=0$ for $1\leq j<m$ and $c_{i-1,m}\neq0$.  Equating coefficients of $t,t^2,\dots,t^{m+r}$, we find $c_{i,j}=0$ for $1\leq j<m+r$ and \[c_{i,m+r}=\frac{f_i^{(r+1)}(p_{i-1})} {(r+1)!}\cdot \frac{c_{i-1,m}^{r+1}}{a_ip_{i}^{a_i-1}}\neq0.\]  Thus, $\ord_P(x_{i}-p_{i})=m+r=r+\ord_P(x_{i-1}-p_{i-1})$.

With the parametrizations, we can now compute the divisors associated to the coordinate functions $x_1,\dots,x_k$ and also $\mathrm{d}x_0$.  We consider the two cases where $P$ is either ramified or unramified.

First, suppose $P=(p_0,p_1,\dots,p_k)$ is ramified in at least one level of the tower.  That is, $P$ has $\ell$ coordinates equal to zero for some $1\leq \ell\leq k$. (Note that we are only considering whether coordinates $p_1,\dots,p_k$ are zero or not.) Then we have indices $1\leq i_1<i_2<\dots<i_\ell\leq k$ and $1\leq j_1<j_2<\dots<j_{k-\ell}\leq k$ such that $p_{i_r}=0$ for $1\leq r\leq \ell$ and $p_{j_s}\neq0$ for $1\leq s\leq k-\ell$.  It will be helpful to keep track of certain products of indices which correspond to zeros in $P$.  Let $Z_r=\prod_{t=r}^{\ell}a_{i_t}$.

Since there are zeros, we must have $f_i'(p_i)\neq0$ for all $i$, so we can ignore Case 3 above. In Cases 1 and 2, $\ord_P(x_i-p_i)$ cannot increase as $i$ increases, so we must have $\ord_P(x_k-p_k)=1$.  Since $\ord_P(x_{i_r-1}-p_{i_r-1})=a_{i_r} \ord_P(x_{i_r})$ and $\ord_P(x_{j_s-1}-p_{j_s-1})=\ord_P(x_{j_s}-p_{j_s})$, we conclude that $\ord_P(x_{i_r})=Z_{r+1}$.  Furthermore, note that $\ord_P(x_0-p_0)=Z_1$.  Since this is not zero, we see that $\ord_P(x_0')=Z_1-1$.

Thus, $\nu_P(x_{i_r})=Z_{r+1}$, $\nu_P(x_{j_s})=0$, and $\nu_P(\mathrm{d}x_0)=Z_1-1$.  Also, note $a_{i_r}\nu_P(x_{i_r})=a_{i_r}Z_{r+1}=Z_r$.  Combining these, we see that 
\begin{align*}
\nu_P(\omega)=\nu_P\left(\frac{\mathrm{d}x_0}{x_1^{a_1-1}x_2^{a_2-2}\cdots x_k^{a_k-1}}\right) &
= \nu_P(\mathrm{d}x_0)-\sum_{j=1}^\ell (a_{i_j}-1)\nu_P(x_{i_j}) \\
&= Z_1-1-\sum_{j=1}^\ell (Z_j-Z_{j+1}) \\
&= 0,
\end{align*} 
as the summation is a telescoping sum.

Next, suppose $P$ is unramified throughout the tower.  That is, $p_1,\dots,p_k\neq0$. Considering only Cases 1 and 3 above, we see $\ord_P(x_i-p_i)$ can only increase as $i$ increases. Thus, $\ord_P(x_0-p_0)=1$, so $\ord_P(\mathrm{d}x_0)=1-1=0$.  Also, $\nu_P(x_i)=0$ for all $1\leq i\leq k$.  Therefore, 
\[\nu_P(\omega)=\nu_P\left(\frac{\mathrm{d}x_0}{x_1^{a_1-1}x_2^{a_2-1}\cdots x_k^{a_k-1}}\right)=0.\]

Thus, for any affine point $P\in C_k$, $\nu_P(\omega)=0$.
\end{proof}

We could find parameterizations at the point at infinity as well, but since we're only missing one point and we know the divisor associated to a differential form on a curve of genus $g$ has degree $2g-2$, we can calculate the divisor associated to $\omega^q$.

\begin{cor}
For any $q\geq1$ and $k\geq1$, $\dv\left(\omega^q\right) = (2g-2)qP_\infty^k$.
\end{cor}

\begin{proof}
By Proposition~\ref{prop:no-affine-points}, there are no affine points in the support of this principal divisor.  Since there is only one non-affine point $P_\infty^k$ in the nonsingular model of $C_k$, and since the principal divisor of a $q$-differential has degree $(2g-2)q$, the associated divisor must be $(2g-2)qP_\infty^k$.
\end{proof}

\subsection{The $q$-weight of the point at infinity}
\label{sec:tower-calculations}
Now that we have found the divisor associated to a particularly nice $q$-differential, we can build a basis of holomorphic $q$-differentials.  

\begin{propn}
A basis for the vector space of holomorphic $q$-differentials on $C_k$ is 
\[\mathfrak{B}_q(A,B):=\left\{ \omega^q\cdot \prod\limits_{i=0}^k x_i^{e_i} : e_0\geq0,\ 0\leq e_i<a_i \text{ for } 1\leq i\leq k,\text{ and } \sum\limits_{i=0}^k e_i g_i \leq (2g-2)q\right\}.\]
\end{propn}

\begin{proof}
Let $h_1,h_2\in \mathfrak{B}_q(A,B)$, so $h_i=\omega^q\prod\limits_{i=0}^k x_i^{e_{\ell,i}}$ for $\ell=1,2$, with $e_{\ell,0}\geq0$ and $0\leq e_{\ell,i}<a_i$ for $1\leq i\leq k$. Suppose $\nu_{\infty,k}(h_1)=\nu_{\infty,k}(h_2)$.  Then $\sum\limits_{i=0}^k e_{1,i} g_i = \sum\limits_{i=0}^k e_{2,i} g_i$.  By Lemma~\ref{lem:R/NR}, with $j=0$, since these sums are equal, their coefficients must be equal, so $h_1=h_2$.  Thus, elements of $\mathfrak{B}_q(A,B)$ have different orders of vanishing at infinity, and so they are linearly independent.

Note also that $\nu_{\infty,k}(h_1)=(2g-2)q-\sum e_{1,i} g_i\geq0$, so $h_1$ is a holomorphic $q$-differential.  Therefore, to prove $\mathfrak{B}_q(A,B)$ is a basis, it remains to show $\# \mathfrak{B}_q(A,B)=d_q$.

Let \[R:=\left\{\sum\limits_{i=0}^k e_i g_i : e_0\geq0\text{ and } 0\leq e_i<a_i\text{ for } 1\leq i\leq k\right\}.\]  By Lemma~\ref{lem:R/NR}, any number representable as a non-negative linear combination of elements of $G(A,B)$ can be written in this form, so $\NR$, the complement of $R$ in $\NZ$, contains exactly $g=S_0(A,B)$ non-representable integers.  The largest element of $\NR$ is $2g-1$.  Also, note that by Corollary~\ref{cor:2g=F+1}, $2g-2\not\in \NR.$

Let $\NR_q:=\{(2g-2)q-s : s\in \NR, (2g-2)q-s\geq0 \}$.  These are the missing orders of vanishing at infinity. If $q=1$, then since $2g-1\in \NR$ and $2g-2\not\in \NR$, we have $\#\NR_1=g-1$.  If $q>1$, by our earlier assumption that $g\ge2$, we have $(2g-2)q>2g-1>0$, so $\#\NR_q=g$.  

Finally, if we let $R_q:=\{(2g-2)q-r : r\in R, (2g-2)q-r\geq0 \},$ note that $\#R_q=\#\mathfrak{B}_q(A,B)=(2g-2)q+1-\#\NR_q$.  If $q=1$, then $\#R_1=(2g-2)+1 - \#\NR_1=g=d_1$.  If $q\geq2$, then $\#R_q=(2g-2)q+1 - \#\NR_q = (g-1)(2q-1)=d_q.$  Thus, $\#\mathfrak{B}_q(A,B)=d_q$ for all $q$, as desired.
\end{proof}

\begin{thm}\label{thm:q-wt-p-infty-compound}
The $q$-Weierstrass weight of $P_\infty^k$ on $C_k$, a curve of genus $g=S_0(A,B)\geq2$, is 
\begin{align*} w^{(q)}(P_\infty^k) = \begin{dcases*}\frac{S_0(A^2,B^2)}{12}-S_0(A,B) & for $q=1$, \\ \frac{S_0(A^2,B^2)}{12} & for $q\geq2$. \end{dcases*}\end{align*}
\end{thm}

\proof
Since the orders of vanishing at $P_\infty^k$ of the basis $q$-differentials are all different, $w^{(q)}(P_\infty^k)$ is the sum of those orders of vanishing minus the sum of the integers from $0$ to $d_q-1$.  The set of orders of vanishing is the set $R_q$ from the above proof. Thus, \[w^{(q)}(P_\infty^k)=\sum\limits_{r_q\in R_q}r_q - \frac{(d_q-1)d_q}{2}.\]  The complement of $R_q$ in the interval $[0,(2g-2)q]$ is $\NR_q=\{(2g-2)q-s : s\in \NR, (2g-2)q-s\geq0\}$, so \[w^{(q)}(P_\infty^k) = \frac{(2g-2)q((2g-2)q+1)}{2} - \sum\limits_{s_q\in \NR_q}s_q - \frac{(d_q-1)d_q}{2}.\]  We now consider the various values of $q$.

If $q=1$, then $\NR_1=\{(2g-2)-s : s\in \NR(A,B)\setminus\{2g-1\} \},$ a set of $g-1$ integers.  Thus, \[\sum\limits_{s_1\in \NR_1}s_1 = \sum\limits_{s\in \NR(A,B)\setminus\{2g-1\}}((2g-2)-s) = (g-1)(2g-2)-(S_1(A,B)-(2g-1)).\] Since $d_1=g$, 
$w^{(1)}(P_\infty^k)= S_1(A,B)-g(g+1)/2.$  By Proposition~\ref{prop:closed-forms}, $w^{(1)}(P_\infty^k)=S_0(A^2,B^2)/12-S_0(A,B)$.

If $q\geq2$, then $\NR_q=\{(2g-2)-s : s\in \NR(A,B)\},$ a set of $g$ integers.  Thus, \[\sum\limits_{s_q\in \NR_q}s_q = \sum\limits_{s\in \NR(A,B)}((2g-2)q-s) = g(2g-2)q-S_1(A,B).\] 
Since $d_q=(g-1)(2q-1)$, $w^{(q)}(P_\infty^k) = S_1(A,B) - g(g-1)/2.$ By Proposition~\ref{prop:closed-forms}, $w^{(q)}(P_\infty^k)=S_0(A^2,B^2)/12$.
\qed

In the special case that $G$ is a geometric sequence of natural numbers with $\gcd(G)=1$, we can solve for $w^{(q)}(P_\infty^k)$ explicitly.

\begin{cor}
For $k$-tuples $A=(a,\dots,a)$ and $B=(b,\dots,b)$ with $\gcd(a,b)=1$ such that $S_0(A,B)\ge 2$,
\begin{align*}
w^{(q)}(P_\infty^k) &= 
\begin{dcases*} 
\frac{(b^2-1)a^{2k+2}-(a^2-1)b^{2k+2}}{24(a^2-b^2)} - \frac{(b-1)a^{k+1}-(a-1)b^{k+1}}{2(a-b)}-\frac{11}{24} & for $q=1$ \\
\frac{(b^2-1)a^{2k+2}-(a^2-1)b^{2k+2}}{24(a^2-b^2)}+\frac{1}{24} & for $q \ge 2$.
\end{dcases*}
\end{align*}
\end{cor}

\subsection{Examples of suitable towers}\label{sec:tower-examples}
There are some examples of nonsingular towers in \cite[Section 3.4]{Shor11}, which are done in finite characteristic.  We use the same approach in characteristic zero.

\begin{thm}
For any $k$-tuples $A=(a_1,\dots,a_k)$ and $B=(b_1,\dots,b_k)$ of natural numbers where $\gcd(a_i,b_j)=1$ for all $i\le j$, and for $c_1,\dots,c_k\in\mathbb{Q}$ such that $c_i$ is not a $b_i$th power of any rational number, let $f_i(x)=x^{b_i}-c_i$ and $H_i(x,y)=y^{a_i}-f_i(x)$.  Then the affine curve \[A_k=\{(x_0,x_1,\dots,x_k)\in\mathbb{C}^{k+1} : H_i(x_{i-1},x_{i})=0\text{ for } 1\leq i\leq k) \}\] is nonsingular for all $k$.
\end{thm}

\proof
The Jacobian matrix of $A_k$ is 
\[J=\begin{pmatrix}
-b_1\cdot x_0^{b_1-1} & a_1\cdot x_1^{a_1-1} & 0 & 0 & \cdots & 0 & 0 \\
0 & -b_2\cdot x_1^{b_2-1} & a_2\cdot x_2^{a_2-1} & 0 & \cdots & 0 & 0 \\
0 & 0 & -b_3\cdot x_2^{b_3-1} & a_3\cdot x_3^{a_3-1} & \cdots & 0 & 0 \\
\vdots & \vdots & \vdots & \vdots & \ddots & \vdots & \vdots \\
0 & 0 & 0 & 0 & \cdots & -b_k\cdot x_{k-1}^{b_k-1} & a_k\cdot x_k^{a_k-1} \\
\end{pmatrix},\] 
which is a $k\times (k+1)$ matrix.  If the rank of $J$ at $P$ is $k$, then $P$ is nonsingular.  Otherwise, $P$ is a singular point.  Note that the rank of $J$ drops precisely when two coordinates of a point are equal to zero.  We will show that this cannot happen on $A_k$.

For the sake of contradiction, suppose $P=(p_0,\dots,p_k)$ has two coordinates equal to zero.  Since our curve is defined iteratively, we may assume $p_0=p_i=0$ for some $i>0$.  Then $p_1$ is an $a_1$th root of $-c_1$, so $[\mathbb{Q}[p_1]:\mathbb{Q}]$ divides $a_1$.  (Also, $p_1\neq0$, so $i\geq2$.)  Similarly, $[\mathbb{Q}[p_2]:\mathbb{Q}]$ divides $a_1 a_2$.  Proceeding in this way, we see $[\mathbb{Q}[p_{i-1}]:\mathbb{Q}]$ divides $a_1 a_2\cdots a_{i-1}$.

Then, since $p_i=0$, we have $0=p_{i-1}^{b_i}-c_i$, which implies $p_{i-1}$ is a $b_i$th root of $c_i$.  Since $c_i$ has no $b_i$th roots in $\mathbb{Q}$, we conclude $[\mathbb{Q}[p_{i-1}]:\mathbb{Q}]$ divides $b_i$ and is greater than 1.  However, since $[\mathbb{Q}[p_{i-1}]:\mathbb{Q}]$ also divides $a_1\cdots a_{i-1}$ and $\gcd(a_1\cdots a_{i-1},b_i)=1$, this is impossible.

Thus, any point $P\in A_k$ can have at most one coordinate equal to zero, which implies $P$ is a nonsingular point of $A_k$.
\qed

In the above theorem, we have the requirement that $\gcd(a_i,b_j)=1$ for all $i\leq j$.  The pair $(A,B)$ is suitable if $\gcd(a_i,b_j)=1$ for all $i\ge j$. Combining these, if $\gcd(a_i,b_j)=1$ for all $i,j$, we can define the curve $C_k$ as in the above theorem and conclude that in its desingularization there is a unique point at infinity $P_\infty^k$ with $q$-Weierstrass weight as given in Theorem~\ref{thm:q-wt-p-infty-compound}.

\bibliography{main-bib-file}{}

\begin{thebibliography}{10}

\bibitem{BrownShiue93}
Tom~C. Brown and Peter Jau-Shyong Shiue.
\newblock A remark related to the {F}robenius problem.
\newblock {\em Fibonacci Quart.}, 31(1):32--36, 1993.

\bibitem{HaggkvistFrobergGottlieb87}
R.~Fr\"{o}berg, C.~Gottlieb, and R.~H\"aggkvist.
\newblock On numerical semigroups.
\newblock {\em Semigroup forum}, 35:63--84, 1987.

\bibitem{KiersONeillPonomarenko16}
Claire Kiers, Christopher O'Neill, and Vadim Ponomarenko.
\newblock Numerical semigroups on compound sequences.
\newblock {\em Comm. Algebra}, 44(9):3842--3852, 2016.

\bibitem{Miura93}
Shinji Miura.
\newblock Algebraic geometric codes on certain plane curves.
\newblock {\em Electronics and Communications in Japan (Part III: Fundamental
  Electronic Science)}, 76(12):1--13, 1993.

\bibitem{Mumford99}
David Mumford.
\newblock {\em The red book of varieties and schemes}, volume 1358 of {\em
  Lecture Notes in Mathematics}.
\newblock Springer-Verlag, Berlin, expanded edition, 1999.
\newblock 

\bibitem{OngPonomarenko08}
Darren~C. Ong and Vadim Ponomarenko.
\newblock The {F}robenius number of geometric sequences.
\newblock {\em Integers}, 8:A33, 3, 2008.

\bibitem{Rodseth93}
{\O}ystein~J. R{\o}dseth.
\newblock A note on {T}. {C}. {B}rown and {P}. {J}.-{S}. {S}hiue's paper: ``{A}
  remark related to the {F}robenius problem'' [{F}ibonacci {Q}uart. {\bf 31}
  (1993), no.\ 1, 32--36; {MR}1202340 (93k:11018)].
\newblock {\em Fibonacci Quart.}, 32(5):407--408, 1994.

\bibitem{RosalesGarciaSanchez09}
J.C. Rosales and P.A. Garc\'{i}a-S\'{a}nchez.
\newblock {\em Numerical Semigroups}, volume~20 of {\em Developments in
  Mathematics}.
\newblock Springer-Verlag New York, 2009.

\bibitem{sagemathcloud}
{SageMath, Inc.}
\newblock {\em SageMathCloud Online Computational Mathematics}, 2016.
\newblock {\tt https://cloud.sagemath.com/}.

\bibitem{ShaskaShor15}
T.~Shaska and C.~Shor.
\newblock Theta functions and symmetric weight enumerators for codes over
  imaginary quadratic fields.
\newblock {\em Des. Codes Cryptogr.}, 76(2):217--235, 2015.

\bibitem{ShaskaShor2015Advances}
T.~Shaska and C.~Shor.
\newblock Weierstrass points of superelliptic curves.
\newblock In L.~Beshaj, T.~Shaska, and E.~Zhupa, editors, {\em Advances on
  Superelliptic Curves and Their Applications}, NATO Science for Peace and
  Security Series - D: Information and Communication Security. IOS Press, 2015.

\bibitem{Shor05}
Caleb~M. Shor.
\newblock {\em On towers of function fields and the construction of the
  corresponding {G}oppa codes}.
\newblock ProQuest LLC, Ann Arbor, MI, 2005.
\newblock Thesis (Ph.D.)--Boston University.

\bibitem{Shor11}
Caleb~M. Shor.
\newblock Genus calculations for towers of functions fields arising from
  equations of {$C_{ab}$} curves.
\newblock {\em Albanian J. Math.}, 5(1):31--40, 2011.

\bibitem{Shor16}
Caleb~M. Shor.
\newblock Higher-order {W}eierstrass weights of branch points on superelliptic
  curves.
\newblock 2016.
\newblock Submitted for publication.

\bibitem{Silverman90}
Joseph~H. Silverman.
\newblock Some arithmetic properties of {W}eierstrass points: hyperelliptic
  curves.
\newblock {\em Bol. Soc. Brasil. Mat. (N.S.)}, 21(1):11--50, 1990.

\bibitem{Stichtenoth09}
Henning Stichtenoth.
\newblock {\em Algebraic function fields and codes}, volume 254 of {\em
  Graduate Texts in Mathematics}.
\newblock Springer-Verlag, Berlin, second edition, 2009.

\bibitem{Sylvester1882}
J.~J. Sylvester.
\newblock On subvariants, i.e. semi-invariants to binary quantics of an
  unlimited order.
\newblock {\em American Journal of Mathematics}, 5(1):79--136, 1882.

\bibitem{Towse96}
Christopher Towse.
\newblock Weierstrass points on cyclic covers of the projective line.
\newblock {\em Trans. Amer. Math. Soc.}, 348(8):3355--3378, 1996.

\bibitem{Tripathi08}
Amitabha Tripathi.
\newblock On the {F}robenius problem for geometric sequences.
\newblock {\em Integers}, 8:A43, 5, 2008.

\bibitem{Tuenter06}
Hans~J.H. Tuenter.
\newblock The {F}robenius problem, sums of powers of integers, and recurrences
  for the {B}ernoulli numbers.
\newblock {\em Journal of Number Theory}, 117(2):376 -- 386, 2006.

\bibitem{WahPicciotto94}
Anita Wah and H.~Picciotto.
\newblock {\em Algebra: Themes, concepts, tools}.
\newblock Creative Publications, 1994.

\end{thebibliography}
\bibliographystyle{plain}

\end{document}